\documentclass[10pt]{amsart}
\usepackage{amsmath,amsthm,amssymb,amscd,graphicx,mathabx}
\usepackage{color}
\usepackage{hyperref}
\begin{document}
\newtheorem{lemme}{Lemma}[section]
\newtheorem{proposition}[lemme]{Proposition}
\newtheorem{cor}[lemme]{Corollary}
\numberwithin{equation}{section}
\newtheorem{theoreme}{Theorem}[section]
\theoremstyle{remark}
\theoremstyle{definition}
\newtheorem{definition}[lemme]{Definition}
\theoremstyle{remark}
\theoremstyle{remark}
\newtheorem{remark}[lemme]{Remark}

\newcommand{\C}{{\mathbb C}}
\newcommand{\N}{{\mathbb N}}
\newcommand{\Z}{{\mathbb Z}}
\newcommand{\R}{{\mathbb R}}
\newcommand{\T}{{\mathbb T}}
\newcommand{\spt}{\,\mathrm{supp}\,}
\renewcommand{\Re}{\;\mathrm{Re}\;}
\renewcommand{\Im}{\;\mathrm{Im}\;}
\title[Nonlinear scattering]
{Almost surely smoothed scattering for cubic NLS}
 \author{N. Burq}
\address{D\'epartment de Math\'ematiques Universit\'e Paris-Saclay, Bat. 307, 91405 Orsay Cedex
France}
\email{nicolas.burq@universite-paris-saclay.fr}
\author{H. Koch} 
\address{Mathematisches Institut, Endenicher Allee 60, Bonn University, Germany} 
\email{koch@math.uni-bonn.de}
 \author{N. Tzvetkov}
\address{ENS Lyon, 46, All\'ee d'Italie, 69 007 Lyon, France}
\email{nikolay.tzvetkov@ens-lyon.fr}
\author{N. Visciglia}
\address{Dipartimento di Matematica, Largo B. Pontecorvo 5, 56124, Universit\`a di Pisa, Italy}
\email{nicola.visciglia@unipi.it}
\date{\today}
 
\begin{abstract} We consider cubic NLS in dimensions $2,3,4$ and we prove
that almost surely solutions with randomized initial data  at low regularity scatter. Moreover, we establish some
smoothing properties of the associated scattering operator.
\end{abstract}
\maketitle

 {\em \small{Msc: 35Q55}}
\section{Introduction and main result}

In this paper we consider the cubic NLS
\begin{equation}\label{NLScubic}
\begin{cases}
i\partial_t u +\Delta u - |u|^2u=0, \quad (t,x)\in \R\times \R^n,\\
u(0,x)=u_0
\end{cases}
\end{equation}
in dimension $n=2,3,4$ and we are interested in the long-time description of the corresponding solutions. The proofs apply to more general situations, but for simplicity we mainly restrict ourselves to those three cases. 
Notice that cubic NLS, depending on the space dimension $n=2,3,4$, is respectively $L^2$-critical, $\dot H^\frac 12$ critical and $\dot H^1$-critical. 

Recall the classical notion of scattering
in $H^s(\R^n)$. Assume that \eqref{NLScubic} with $u_0\in H^s(\R^n)$ 
admits an unique global solution $u(t,x)$ belonging in some functional space embedded in ${\mathcal C}(\R;H^s(\R^n))$, then we say that $u(t,x)$ scatters in $H^s(\R^n)$ as  $t\rightarrow \pm \infty$ provided that:
\begin{equation}\label{Hssca}
\exists u_0^\pm \in H^s(\R^n) \hbox{ s.t. } \|e^{-it\Delta} u(t,x)- u_0^\pm\|_{H^s(\R^n)}
\overset{t\rightarrow \pm \infty}\longrightarrow 0.
\end{equation}
In the context of cubic NLS we recall that scattering in $H^s(\R^n)$ has been proved in \cite{Do} for any $s\geq 0$ and $n=2$, in \cite{GV} for any $s\geq 1$ and $n=3$, in \cite{RV} for any $s\geq 1$ and $n=4$.  It is worth mentioning that in dimension $n=3$ it is proved in \cite{CKSTT} that scattering 
occurs also in $H^s(\R^n)$ for a range of $s\in (s_0, 1)$ with $0<s_0<1$.
 For a general overview of the scattering issue for NLS we also quote the book \cite{Caz}
and the references therein.
In any case we point out that according to \eqref{Hssca} if the initial datum $u_0$ belongs to $H^s(\R^n)$, then the scattering stats $u_0^\pm$ have the same regularity and moreover the convergence of nonlinear waves to free 
waves occurs in the same topology.

In last decades a probabilistic approach allowed to give insight into the dynamics of nonlinear PDEs even for Sobolev regularity below the one allowed by the deterministic theory.
This point of view started in \cite{LRS}, has been further developed by \cite{B1}, \cite{B2}, and
then by many authors,  we quote \cite{BT2}, \cite{BT1},  \cite{CO}, \cite{NS} to mention a few of them. More recently, starting from \cite{BTT}, also the scattering theory has been enhanced 
thanks to the probabilistic approach. Typically, one can show, a.s. with respect to suitable non-trivial measures on the phase space of initial data, that the scattering operator $u_0\rightarrow u_0^\pm$
has a more precise structure, namely $u_0^\pm$ is a smoother perturbation of the initial datum $u_0$, and moreover the asymptotic convergence in \eqref{Hssca} occurs on a topology
stronger than the one of the initial datum.
This point of view has been developed, for example, in \cite{BT}, \cite{BTT}, \cite{C}, \cite{DLM}, \cite{Ma}, \cite{PRT}, \cite{SSW} and \cite{Sp}. In the sequel we shall follow the approach introduced in \cite{PRT}, whose randomization procedure was inspired by \cite{BL}.
\subsection{The randomization procedure and main result} In order to state our result, let us introduce the harmonic Sobolev spaces ${\mathcal H}^s(\R^n)$
associated with the following norm:
\begin{equation}
\|u\|_{{\mathcal H}^s(\R^n)}=\| H^\frac s2 u \|_{L^2(\R^n)},
\end{equation}
where $H$ denotes the harmonic oscillator, namely
\begin{equation}\label{harmosc}H=-\Delta + |x|^2.\end{equation}
Next, following \cite{PRT}, we introduce some probability measures ${\mathcal M}_n^s$ supported in the phase space ${\mathcal H}^s(\R^n)$.
Given $\{\lambda_k, k\geq 1\}$ the family of eigenvalues of $H$, repeated with their multiplicity and ordered in an increasing sequence, we introduce 
$$I(j)=\{k\in \N \hbox{ s.t. } 2j\leq \lambda_k<2(j+1)\}.$$
We also introduce the following subset of ${\mathcal H}^s(\R^n)$: 
\begin{multline}\label{As}{\mathcal A}^s(\R^n)=\Big \{\sum_{j\in \N} \sum_{k\in I(j)} c_k \varphi_k \in {\mathcal H}^s(\R^n)
\hbox{ s.t. } \\ |c_k|^2\lesssim \frac{1}{\# I(j)} \sum_{m\in I(j)} |c_m|^2, \forall k\in I(j), j\geq 1\Big \}
\end{multline} 
where $\{\varphi_k, k\in \N\}$ is an orthonormal basis of $L^2(\R^n)$ of eigenfunctions for $H$ associated with the eigenvalues $\lambda_k$.
Next we fix a probability space $(\Omega, {\mathcal F}, \mathbb P)$ and a family $g_k:\Omega\rightarrow \R$ of i.i.d. random variables, whose common law $\nu$ satisfies $\int_\R e^{cx} d\nu\leq e^{\kappa c^2}$ for every $c\in \R$ and some fixed $\kappa >0$. As a consequence the random variables $g_k$ are centered. Then for any  $\gamma\in {\mathcal A}^s(\R^n)$
we can define the probability measure $\mu_\gamma$ on $\mathcal H^s(\R^n)$ as the 
push-forward of the measure $(\Omega, {\mathcal F}, \mathbb P)$ via the map
$$\Omega\ni \omega\rightarrow \sum_{j\in \N} \sum_{k\in I(j)} c_k g_k(\omega) \varphi_k.$$
Notice that the map above takes value
in 
${\mathcal H}^s(\R^n)$$\mathbb P$ a.s. We can now introduce a family of measures on the phase space $\mathcal H^s(\R^n)$ that will be used along the rest of paper:
\begin{equation}
{\mathcal M}_n^s=\bigcup_{\gamma\in {\mathcal A}^s(\R^n)} \mu_\gamma.
\end{equation}
We have that a measure in ${\mathcal M}_n^s$ depends on the choice of the basis  $\{\varphi_k, k\in \N\}$. In addition,  one can show the following properties about the measures belonging to 
${\mathcal M}_n^s$ (see \cite{PRT}):
\begin{itemize}
\item assume $\gamma \in {\mathcal A}^s(\R^n)\setminus {\mathcal H}^{s+\varepsilon}(\R^n)$, then 
\begin{equation}\label{noderiv}
\mu_\gamma ({\mathcal H}^{s+\varepsilon}(\R^n))=0;\end{equation}
\item assume that $\gamma\in {\mathcal A}^s(\R^n)$, the corresponding Fourier coefficients $c_k$ of $\gamma$ (see \eqref{As}) are different from zero and the support of the law of $g_k$ is $\R$, then 
\begin{equation}\label{nontrivialmessure}
\mu_\gamma(B)>0, \quad \forall B \hbox{ open set in }  {\mathcal H}^s(\R^n).
\end{equation}
\end{itemize}
We now come to the main result of the paper.
\begin{theoreme}\label{4d}
Let $(n, s)\in \{2,3,4\}\times \R$ satisfy 
\begin{equation}\label{eq:sn} \left\{ \begin{array}{cl}   s>0 & \text{ if } n=2 \\
s\ge-\frac14 & \text{ if } n= 3\\
s\ge -\frac12 & \text{ if } n= 4 
\end{array}\right.
\end{equation} 
 then for every $\mu_\gamma\in {\mathcal M}^s_n$
there exists a set $\Sigma_n^s\subset {\mathcal H}^s(\R^n)$ with full measure w.r.t. $\mu_\gamma$
such that 
for every $u_0\in \Sigma_n^s$
the corresponding Cauchy problem \eqref{NLScubic} admits a unique global solution with the following structure
\begin{equation}\label{struc}
u(t,x)=e^{it\Delta} (u_0+ R_{u_0}(t,x)), \quad R_{u_0}(t,x)\in {\mathcal C}(\R; {\mathcal H}^1(\R^n)).
\end{equation} Moreover there exist 
\begin{equation}\label{trace} \lim_{t\to \pm \infty} R_{u_0}(t)=r_0^{\pm}\in {\mathcal H}^1(\R^n)\end{equation}  such that
\begin{equation}\label{convin}\|e^{-it\Delta} u(t,x)-u_0 -r_0^{\pm}\|_{{\mathcal H}^1(\R^n)}\overset {t\rightarrow \pm \infty} \longrightarrow 0.\end{equation}
More precisely we get as $t\rightarrow \pm \infty$ the following behavior:
\begin{equation} \label{eq:convergence}    \Vert e^{-it \Delta} u(t,x) -u_0 - r_0^{\pm} \Vert_{{\mathcal H}^1(\R^n)} = \left\{\begin{array}{cl}  O(|t|^{-1+\delta}) \quad \forall \delta>0 & \text{ if } n=2 \\
                          O(|t|^{-\frac32})  & \text{ if } n=3 \\
                          O(|t|^{-2}) & \text{ if } n=4. \end{array}\right.
                          \end{equation}    
\end{theoreme}
\begin{remark} 
The proof applies to more general situations: 
\[ i \partial_t u + \Delta u = -|u|^{p-2} u \] 
for $ 2+\frac4n \le p \le 4$ in dimension $n=2,3,4$, the energy critical case $p= \frac{2n}{n-2}$ for $ n \ge 5$ and the energy subcritical case $2+\frac4n\le p <2+ \frac{4}{n-2}$ for $ n \ge 4$.  The lower bound for the power $p$ is optimal in view of the results \cite{IW}. For simplicity we focus on $p=3$ and $n=2,3,4$ and indicate the changes for the more general cases in remarks.
\end{remark} 

Our proof of Theorem~\ref{4d} also shows that our solutions can be obtained as limits of smooth solutions of \eqref{NLScubic} . More precisely, for $u_0\in \Sigma_n^s$, if $u_0^N = P_N u_0$ ($P_N$ being a spectral truncation of $u_0$ based on $H$)  and $u^N$ is the solution of \eqref{NLScubic} with initial data $u_0^N$ then 
\[ u^N \to u \qquad \text{ in } {\mathcal C}(\R, \mathcal{H}^s), \] 
more precisely 
\[  e^{-it \Delta} u^N(t)- u^N_0 \to   e^{-it\Delta } u(t)- u^0 \] 
in ${\mathcal C}(\R; \mathcal{H}^1(\R^n))$. The proof of this claim can be done as in \cite{Tz} (Theorem~2.7 of Chapter~2)  where a similar proof is given in the context of the nonlinear wave equation. 

Notice that in dimension $n=3,4$ we get scattering for data at negative regularity. Moreover, by \eqref{convin} we deduce that the convergence to the asymptotic states occurs in the strong topology ${\mathcal H}^1(\R^n)$.
Concerning our result for $n=2$ we point out that, although we don't treat supercritical initial data (recall that the $L^2$ deterministic scattering theory cubic NLS on $\R^2$ has been achieved in \cite{Do}),
interestingly we still get the strong convergence provided in \eqref{convin}. 

Theorem \ref{4d} extends in several directions, in the specific case of cubic NLS, some results proved in \cite{PRT}.
Indeed for $n=2$ our result is not new compared with \cite{PRT}, however we believe that our argument, based on a suitable energy estimate, provides a rather simple proof. In the case $n=3$, as already said, we can treat  data at negative regularity.
Finally in dimension $n=4$, when the equation is energy critical, as far as we can see the high-low method 
used in \cite{PRT} does not apply for globalization and hence a different argument is required. We also believe that our technique is rather robust and can be applied in a variety of situations including the case where the nonlinearity is not analytic. 

Let us finally mention that Theorem~\ref{4d} is not easily comparable with the results obtained in \cite{C}, \cite{DLM}, \cite{SSW}. From one hand the regularity assumptions in  Theorem~\ref{4d} are weaker compared to \cite{C}, \cite{DLM}, \cite{SSW}. 
On the other hand the localisation assumptions on the data in  Theorem~\ref{4d} are stronger compared to \cite{C}, \cite{DLM}, \cite{SSW}. 

\subsection{Lens transform and strategy of proof} Next we explain our strategy of proof. 
Following \cite{PRT}, the idea to establish Theorem \ref{4d} is to reduce the problem, via the lens transform, to the analysis
of suitable non autonomous cubic NLS on a finite strip of time.
Indeed let us associate with $u(t,x)$ solution to \eqref{NLScubic} the lens transform defined as follows:
\begin{equation}\label{eq:lens}
w(\tau,y)= \frac 1{\cos^\frac n2(2\tau)}
u\big (t(\tau), \frac y{\cos(2\tau)}\big )e^{-i|y|^2 t(\tau)}, \quad (\tau,y)\in (-\frac \pi 4, \frac \pi 4)\times \R^n
\end{equation}
where $$t(\tau)=\frac{\tan(2\tau)}{2}.$$
One can show that $w(\tau,y)$
solves the following nonlinear Cauchy problem:
\begin{equation}
\begin{cases}\label{lenstr} i\partial_\tau w- H w -\cos^{n-2}(2\tau) |w|^{2}w=0, \quad (\tau,y)\in (-\frac \pi 4,\frac \pi 4)\times \R^n\\
 w(0)=u_0, 
\end{cases}
\end{equation}
where $H=-\Delta + |y|^2$. 
We shall use also the following relation between $u$ and $w$
connected via \eqref{eq:lens} for a fixed $\tau$
\begin{equation}\label{exp}
e^{-it(\tau)\Delta} u(t(\tau))= e^{i\tau H} w(\tau), \quad \forall \tau.\end{equation}
In order to show this identity first notice that the lens transform of a solution to the linear Schroedinger equation solves  
the linear Schroedinger equation with harmonic oscillator with the same initial condition (notice that the lens transform is the identity at time $0$), 
and this fact readily implies by the definition \eqref{eq:lens} of lens transform
\begin{equation*}
M_{t(\tau)}\circ S_{\frac 1{\cos(2\tau)}}\circ e^{it(\tau)\Delta}=e^{-i\tau H},
\quad \forall \tau
\end{equation*}
where 
$S_\lambda f(x)=\lambda^\frac n2 f(\lambda x), \quad M_\mu f(x)=e^{-i\mu |x|^2}$. This implies identity \eqref{exp}.

Next notice that the solution $w$ to \eqref{lenstr} can be written as 
\begin{equation}\label{eq:decomp}w(\tau)=e^{-i\tau H} u_0+v(\tau)\end{equation}
where $v$ solves 
the following Cauchy problem
\begin{equation}
\begin{cases}\label{enery} i\partial_\tau v- H v = \cos^{n-2}(2\tau) |e^{-i\tau H}u_0+ v|^{2}( e^{-i\tau H}u_0 +v),
\\v(0)=0
\end{cases}
\end{equation}
on the strip $(\tau,y)\in (-\frac \pi 4,\frac \pi 4)\times \R^n$. Moreover we get
by \eqref{exp} and \eqref{eq:decomp}
\begin{equation}\label{confrel}e^{-it(\tau)\Delta} u(t(\tau))=u_0 +e^{i\tau H} v(\tau).
\end{equation}
Hence \eqref{struc} follows provided that $v(\tau)\in {\mathcal C}
((-\frac \pi 4, \frac \pi 4); {\mathcal H}^1(\R^n))$, indeed in this case we can select
$R_{u_0}(t)= e^{i\tau(t)H} v(\tau(t))$,
where $\tau(t)=\frac{\arctan(2t)}{2}$.
Moreover \eqref{trace}, \eqref{convin} 
follow provided that we prove:
\begin{equation}
v(\tau ,y)\overset{\tau \rightarrow {\pm \frac \pi 4}} \longrightarrow v^\pm \hbox{ in } {\mathcal H}^1(\R^n),
\end{equation}
in fact in this case we get \eqref{trace} and \eqref{convin} by choosing $r_0^\pm=e^{\pm i \frac \pi 4 H}v^\pm$.\\
Notice also that by \eqref{confrel} we get
\begin{equation}\label{imp...}e^{-it(\tau)\Delta} u(t(\tau))-u_0-r_0^\pm=e^{i\tau H} v(\tau)-e^{\pm i \frac \pi 4 H}v^\pm\end{equation}
and hence
we have that if for some $\mu>0$ 
\begin{equation}\label{rate}\|e^{i\tau H} v(\tau)-e^{\pm i \frac \pi 4 H}v^\pm\|_{{\mathcal H}^1(\R^n)}=O(|\pm \frac \pi 4 -\tau|^\mu) \hbox{ as } \tau\rightarrow \pm \frac \pi 4^\mp\end{equation}
then by \eqref{imp...}
\begin{multline*}\|e^{-it\Delta} u(t)-u_0-r_0^\pm\|_{{\mathcal H}^1(\R^n)}=O\big(\big|\pm \frac \pi 4 -\tau(t)\big|^\mu\big)\\=
O\big(\big|\pm \frac \pi 4 -\frac{\arctan 2t}2\big|^\mu\big)= O(|t|^{-\mu})\hbox{ as } t\rightarrow \pm \infty\end{multline*}
which corresponds to \eqref{eq:convergence} for suitable $\mu$ depending on the space dimension $n=2,3,4$.

The rest of the paper is devoted to the analysis of the Cauchy problem
\eqref{enery}. In particular we aim at proving the existence of an unique solution belonging to the space 
$v\in 
{\mathcal C}([-\frac \pi 4, \frac \pi 4]; {\mathcal H}^1(\R^n))$ (notice we require the solution $v$ to exist up to $\tau=\pm \frac \pi 4$) which guarantees 
\eqref{struc}, \eqref{trace}, \eqref{convin} once we go back from $v$ to
$w$ and from $w$ to $u$ via the inverse lens transform.
Moreover, based on the discussion above, in order to deduce
\eqref{eq:convergence} we shall show 
\eqref{rate} with $\mu=1-\delta$ for $n=2$, $\mu=\frac 32$ for $n=3$, 
$\mu=2$ for $n=4$.

In order to deal with \eqref{enery} we shall establish first a uniform energy estimate for the quantity ${\mathcal H}^1(\R^n)$. In turn this bound will allow to conclude easily in the case $n=2,3$ since the corresponding cubic NLS is energy subcritical.
However more work is needed in the case $n=4$ since in this case 
the ${\mathcal H}^1(\R^4)$ bound is not sufficient for the globalization up to $\tau=\pm \frac \pi 4$ and suitable Strichartz norms need to be controlled as well.

It is worth mentioning that in order to deal with \eqref{enery} we shall assume suitable smoothing properties of the free waves 
$e^{-i\tau H} u_0$ which are valid $\mu_\gamma$-a.s. granted by stochastic bounds established in \cite{PRT} and extended in section \ref{stoch} of this paper. 
At the best of our knowledge the idea to globalize solutions 
via a Grönwall argument, by taking advantage of stochastic smoothing  
properties of free waves, has been first introduced in \cite{BT1} for Nonlinear 
Wave Equations. Compared with \cite{BT1} we have in this paper two difficulties, on one hand the energy 
estimate is more involved in the context of NLS, on the other hand in dimension $n=4$ we have 
to deal with an energy critical problem.

We finally recall that Strichartz estimates are available for the propagator 
$e^{-itH}$, namely we have the bound:
\begin{equation}\label{strichnoD}
\|u\|_{L^p((-\frac \pi 4, \frac \pi 4);L^q(\R^n))}
\leq C \|u(0)\|_{L^2(\R^n)} + C 
\|i\partial_t u - H u\|_{L^{\tilde p'}((-\frac \pi 4, \frac \pi 4);L^{\tilde q '}(\R^n))}
\end{equation}
provided that 
\begin{equation}\label{condpq}\frac 2 p + \frac nq=
\frac 2{\tilde p} + \frac n{\tilde q}=\frac n2, \quad p\geq 2, \quad (n,p)\neq (2,2),  (n,\bar p)\ne (2,2). 
\end{equation}
Indeed the proof \eqref{strichnoD} 
follows since $w=e^{i\tau H}\varphi$
is the lens transform of the propagator $u=e^{it\Delta} \varphi$ and by direct computation from
\eqref{eq:lens} we get 
\begin{multline*}\|w(\tau)\|_{L^\infty(\R^n)}
=\frac 1{\cos^\frac n2 (2\tau)}\|u(t(\tau)\|_{L^\infty(\R^n)}
\\\leq \frac C{(\cos(2\tau) |t(\tau)|)^\frac n2}\|\varphi\|_{L^1(\R^n)}=
\frac C{|\sin(2\tau)|^\frac n2}\|\varphi\|_{L^1(\R^n)}
\end{multline*}
where we have used the bound 
$\|e^{it\Delta}\varphi\|_{L^\infty(\R^n)}
\leq \frac{C}{|t|^\frac n2}\|\varphi\|_{L^1(\R^n)}$.
Hence we obtain 
$$\|e^{i\tau H}\varphi\|_{L^\infty(\R^n)}
\leq \frac{C}{|\tau|^\frac n2}\|\varphi\|_{L^1(\R^n)}, \quad \tau\in (-\frac \pi 4, \frac \pi 4)$$
and the proof of \eqref{strichnoD} follows via classical arguments.

Along the paper we shall work with 
Strichartz estimates for derivatives:
\begin{multline}\label{strichD}
\|u\|_{L^p((-\frac \pi 4, \frac \pi 4);{\mathcal W}^{1,q}(\R^n))}
\leq C \|u(0)\|_{{\mathcal H}^1(\R^n)} \\+ C 
\|i\partial_t u + H u\|_{L^{\tilde p'}((-\frac \pi 4, \frac \pi 4);{\mathcal W}^{1,\tilde q '}(\R^n))}
\end{multline}
where $p, q, \tilde p, \tilde q$ are as in \eqref{condpq} and for general $s\in \R$ we have denoted $$\|u\|_{{\mathcal W}^{s,q}(\R^n)}=\|H^{\frac s2} u\|_{L^q(\R^n)},$$
hence in the specific case $q=2$ we get the space ${\mathcal H}^s(\R^n)$.
Notice that \eqref{strichD} follows from
\eqref{strichnoD} since
$e^{itH}$ commutes with $(-H)^{\frac s2}$.

We shall make use along the paper of the fact that for every $ s\geq 0, q\in (1, \infty)$ there exists $C>0$ such that
\begin{multline}\label{equivspaces}
\frac 1C \big(\| \langle x \rangle ^s u\|_{L^q(\R^n)}+\| D^s u\|_{L^q(\R^n)} 
\big) \leq \|u\|_{\mathcal W^{s,q}(\R^n)} \\\leq C  \big(\| \langle x \rangle ^s u\|_{L^q(\R^n)}+\| D^s u\|_{L^q(\R^n)} 
\big)
\end{multline}
where $D^s$ is defined by the Fourier multiplier $\langle \xi \rangle^s$. 

\section{Stochastic bounds}\label{stoch}

We shall need some stochastic bounds associated with the propagator $e^{-itH}$ which are available $\mu_\gamma$ almost surely, where $\mu_\gamma\in {\mathcal M}_n^s$ is one of the measures defined  the introduction.
Recall that by Proposition 2.1 in \cite{PRT}
we have $\mu_\gamma$ almost surely the bound
\begin{equation}\label{eq:khinchin}  \Big\Vert e^{it H} u_0 \Big\Vert_{L^q([-\frac\pi4.\frac\pi4];{\mathcal W}^{s+\frac n2-\frac nq,q} (\R^n))}<\infty.\end{equation} 
More precisely 
\begin{equation}  \mathbb{P} \Big( \Big\{\Big\Vert e^{-it H} u_0 \Big\Vert_{L^q([-\frac\pi4,\frac\pi4];\mathcal{ W}^{s+\frac n2-\frac nq,q} (\R^n))} > K \Big\}\Big) \le  c \exp\Big(-C \frac{K^2}{ T^{\frac2q} \Vert \gamma \Vert^2_{\mathcal{H}^s}}\Big).  \end{equation}  

The main result of the section is the following one.
\begin{proposition}\label{sto}
Let $\eta>0$ and $(n, s)\in \{2,3,4\}\times \R$ satisfy 
\begin{equation}\label{eq:sn2} \left\{ \begin{array}{cl}   s>0 & \text{ if } n=2 \\
s>-\frac12 & \text{ if } n= 3\\
s> -1 & \text{ if } n= 4 
\end{array}\right.
\end{equation}
then for every $\mu_\gamma\in {\mathcal M}^s_n$
there exists a Borel set $\Sigma_n^s\subset {\mathcal H}^s(\R^n)$ such that:
\begin{equation}
\mu_\gamma(\Sigma_n^s)=1;
\end{equation}
we have the bound
\begin{multline} \Big\Vert \langle x \rangle^{s+\frac{n}2-1-\eta} \nabla e^{itH} u_0 \Big\Vert_{L^\infty(\R\times \R^n)}
\\+ \Big\Vert \langle x \rangle^{s+\frac{n}2-\eta}  e^{it H} u_0 \Big\Vert_{L^\infty(\R\times \R^n)}
<\infty, \quad \forall u_0\in \Sigma_n^s.
\end{multline}

\end{proposition}

In order to prove Proposition \ref{sto} we shall need the following Lemma.

\begin{lemme}\label{Linftynew} Let $\rho\geq 0$, $n\in \N$,
 $\varepsilon\in (0, \infty)$ and $q\in [1,\infty)$
 such that $\varepsilon q>n$. Then there exists $C>0$ such that:
 \begin{equation}\label{LPS}
 \|\langle x \rangle^{\rho} \nabla u\|_{L^\infty(\R^n)}
 +\|\langle x \rangle^{\rho+1} u\|_{L^\infty(\R^n)} \leq C \|u\|_{{\mathcal W}^{1+\rho+\varepsilon, q}(\R^n)}.\end{equation}
 
 \end{lemme}
 
 \begin{proof}
  We denote by $\Delta_N$ the classical Littlewood-Paley partition where 
 $N$ runs in the set of dyadic numbers such that $N\geq 1$ and 
 we also fix a cut-off function $\varphi\in C^\infty(\R^n; [0,1])$ such that
 $\varphi(x)=1$ on $\frac 12<|x|<2$
 and $\varphi(x)=0$ on $\R^n \setminus \{\frac 14<|x|<4\}.$ 
 We first estimate $\|\langle x \rangle^{\rho+1} u\|_{L^\infty(\R^n)} $.
By the Minkowski inequality and by the Sobolev embedding
$W^{\delta, q}(\R^n)\subset L^\infty(\R^n)$ for some $\delta\in (0, \varepsilon)$ such that $\delta q>n$ (below $W^{s, r}$ denote the usual Sobolev spaces),  we get
\begin{multline*}\big \|\varphi\big(\frac xM\big) u\big \|_{L^\infty(\R^n)}\leq C \sum_N \big \|\Delta_N(\varphi\big(\frac xM\big)u)\big \|_{W^{\delta, q}(\R^n)}\\
\leq C \sum_{N<M} N^{\delta} \big \|\varphi\big(\frac xM\big)u\big \|_{L^{q}(\R^n)}
\\+ C\sum_{N\geq M} N^{-1-\varepsilon-\rho+\delta}\big \|\Delta_N(\varphi\big(\frac xM\big)u)\big \|_{W^{1+\rho+\varepsilon, q}(\R^n)}\\\leq 
C M^\delta\big \|\varphi\big(\frac xM\big)u\big \|_{L^{q}(\R^n)}
+ CM^{-\rho-1}\sum_{N\geq M} N^{-\varepsilon+\delta}\big \|\varphi\big(\frac xM\big)u\big \|_{W^{1+\rho+\varepsilon, q}(\R^n)}.
\end{multline*}
Summarizing we get
\begin{multline}\label{albplu}
M^{1+\rho} \big \|\varphi\big(\frac xM\big) u(x)\big \|_{L^\infty(\R^n)}\\
\leq 
C M^{1+\rho+\delta} \big \|\varphi\big(\frac xM\big)u \big\|_{L^q(\R^n)}
+ C\big \|\varphi \big(\frac xM\big)u\big \|_{W^{1+\rho+\varepsilon, q}(\R^n)}
\\\leq  C \big \|\langle x \rangle^{1+\rho+\delta} u \big\|_{L^q(\R^n)}
+ C \|u\|_{W^{1+\rho+\varepsilon, q}(\R^n)}
\end{multline}
where we used that 
$\sup_M \big \|\varphi \big(\frac xM\big)u\big \|_{W^{1+\rho+\varepsilon, q}(\R^n)}\leq 
C \|u\|_{W^{1+\rho+\varepsilon, q}(\R^n)}$.
It is now easy to deduce, by using the localization property of the function $\varphi$, the bound:
$$
\|\langle x \rangle^{\rho+1} u\|_{L^\infty(\R^n)} \leq C \|u\|_{{\mathcal W}^{1+\rho+\varepsilon, q}(\R^n)}.$$ 
Next we estimate $\|\langle x \rangle^{\rho} \nabla u\|_{L^\infty(\R^n)}$.
First we notice that 
\begin{equation}\label{LEIB}\varphi\big (\frac xM\big )\nabla u=\nabla \big[\varphi\big (\frac xM\big ) u\big] - \frac 1M \nabla \varphi \big (\frac xM\big ) u.\end{equation}
Mimicking the proof of the first inequality in \eqref{albplu} (with $\varphi$ replaced by $\nabla \varphi$)
we get
\begin{multline*}
M^{1+\rho} \big \|\nabla \varphi\big(\frac xM\big) u(x)\big \|_{L^\infty(\R^n)}
\\\leq  C M^{1+\rho+\delta} \big \|\nabla \varphi\big(\frac xM\big)u \big\|_{L^q(\R^n)}
+ C\big \|\nabla \varphi \big(\frac xM\big)u\big \|_{W^{1+\rho+\varepsilon, q}(\R^n)}
\end{multline*}
and hence, by using the localization property of $\varphi$, we get the following bound for the second term in the r.h.s. of \eqref{LEIB}:

\begin{equation}\label{secondrhs}
M^{2+\rho} \big \|\frac 1M \nabla \varphi\big(\frac xM\big) u(x)\big \|_{L^\infty(\R^n)}
\leq C \big \|\langle x \rangle^{1+\rho+\delta} u \big\|_{L^q(\R^n)}
+ C \|u\|_{W^{1+\rho+\varepsilon, q}(\R^n)}.
\end{equation}
We estimate below by the Minkowski inequality and the Sobolev embedding the first term on the r.h.s. of \eqref{LEIB}:
\[ 
\begin{split} \big \|\nabla \big[\varphi\big(\frac xM\big) u\big]\big \|_{L^\infty(\R^n)}\, & \leq C \sum_N \big \|\nabla \Delta_N(\varphi\big(\frac xM\big)u)\big \|_{W^{\delta, q}(\R^n)}
\\ & \leq C \sum_{N<M} N^{1+\delta} \big \|\varphi\big(\frac xM\big)u\big \|_{L^{q}(\R^n)}
\\ & \qquad + C\sum_{N\geq M} N^{-\varepsilon-\rho+\delta}\big \|\Delta_N(\varphi\big(\frac xM\big)u)\big \|_{W^{1+\rho+\varepsilon, q}(\R^n)}\\  &\hspace{-2cm}  \leq 
C M^{1+\delta}\big \|\varphi\big(\frac xM\big)u\big \|_{L^{q}(\R^n)}
+ CM^{-\rho}\sum_{N\geq M} N^{-\varepsilon+\delta}\big \|\varphi\big(\frac xM\big)u\big \|_{W^{1+\rho+\varepsilon, q}(\R^n)},
\end{split}
\] 
where $0<\delta<\varepsilon$ satisfies $\delta q>n$.
In turn this estimate implies
\[ \begin{split} \hspace{2cm} & \hspace{-2cm} M^\rho \big \|\nabla \big[\varphi\big(\frac xM\big) u\big]\big \|_{L^\infty(\R^n)}
\\ & \leq C M^{1+\rho+\delta}\big \|\varphi\big(\frac xM\big)u\big \|_{L^{q}(\R^n)}
+ C\sum_{N\geq M} N^{-\varepsilon+\delta}\big \|\varphi\big(\frac xM\big)u\big \|_{W^{1+\rho+\varepsilon, q}(\R^n)}
\\ & \leq  C \big \|\langle x \rangle^{1+\rho+\delta} u \big\|_{L^q(\R^n)}
+ C \|u\|_{W^{1+\rho+\varepsilon, q}(\R^n)}.
\end{split}\] 
By combining this estimate with \eqref{LEIB} and \eqref{secondrhs}, and by recalling the localization property of $\varphi$, 
we get
$$\|\langle x \rangle^{\rho} \nabla u\|_{L^\infty(\R^n)} \leq C \|u\|_{{\mathcal W}^{1+\rho+\varepsilon, q}(\R^n)}.$$

 \end{proof}
 
 \begin{proof}[Proof of Prop. \ref{sto}]
 We fix $\varepsilon$, $q$ and $\rho$ as in Lemma \ref{Linftynew}.
 Then we get 
 \begin{multline*} 
 \big \| \langle x \rangle^{\rho} \nabla e^{-itH} u_0 \big\|_{L^\infty(\R\times \R^n)}
+ \big \| \langle x \rangle^{\rho+1}  e^{-it H} u_0 \big \|_{L^\infty(\R\times \R^n)}
\\\leq C \big \|  e^{-it H} u_0 \big\|_{L^\infty(\R; {\mathcal W}^{1+\varepsilon+\rho, q} (\R^n))}
\leq C \big \|  \langle H \rangle^{\frac{1+\varepsilon+\rho}2} e^{-it H} u_0 \big\|_{L^q(\R^n;L^{\infty} (\R))}
\end{multline*}
where we used \eqref{LPS} and the Minkowski inequality at the last step to interchange space and time in the mixed space-time norm.
Next we use the Sobolev embedding for functions periodic in time
$${\mathcal W}^{\varepsilon, q}_{per}(\R)\subset L^\infty(\R)$$
and we can continue the estimate above as follows
$$\dots \leq C \Big\Vert  \langle D_t\rangle^\varepsilon \langle H \rangle^{\frac{1+\varepsilon+\rho}2} e^{-it H} u_0 \Big\Vert_{L^q(\R;L^q (\R^n))}
=C\Big\Vert \langle H \rangle^{\frac{1+3\varepsilon+\rho}2} e^{-itH} u_0 \Big\Vert_{L^q(\R;L^q (\R^n))}$$
where we used the property
\begin{equation}\label{fractcal}\langle D_t\rangle^\varepsilon \xi=\langle H \rangle^\varepsilon \xi.
\end{equation}
For completeness we prove this fact below, first we conclude the proof.

 We conclude by \eqref{eq:khinchin}  provided that 
$1+\rho+3\varepsilon<s+\frac n2-\frac nq$.
Recall that we used in the computations above
$\varepsilon q>n$, hence we can replace $\varepsilon$ by
$\varepsilon_K=\frac \varepsilon K$ and $q$ by $q_K=qK$ and the estimate above is still satisfied. Hence if we choose
$K$ large enough we have that the condition 
$1+\rho+3\varepsilon_K<s+\frac n2-\frac n{q_K}$
is satisfied for $K$ large enough
provided that $s, n$ satisfy $\rho<s+\frac n2 -1$.
On the other hands in the Lemma \ref{Linftynew} 
we have the condition $\rho\geq 0$ and hence we have to impose
$s+\frac n2 -1>0$, which is equivalent to the conditions
imposed on $s,n$ along Proposition \ref{sto} and the proof is complete.

Next we prove \eqref{fractcal}.
Recall that for a generic $\psi(x)=\sum_k \psi_k \varphi_k(x)$, where $\varphi_k(x)$ is basis of eigenfunctions for $H$ with eigenvalues $\lambda_k$,
we have
$$\langle H \rangle^\varepsilon \psi= \sum_k \psi_k  \langle \lambda_k \rangle^\varepsilon \varphi_k(x)$$
and similarly for a function $\phi(t)=\sum_k \phi_k e^{-ikt}$
we have
$$\langle D_t \rangle^\varepsilon \phi= \sum_k \phi_k \langle k \rangle^\varepsilon e^{ikt}.$$
On the other hands for $u_0=\sum_k c_k \varphi_k(x)$ we have
$$e^{-itH} u_0= \sum_k c_k \varphi_k(x) e^{-it\lambda_k}$$
and hence
\begin{equation*}
\langle H \rangle^\varepsilon e^{-itH} u_0= 
\sum_k c_k e^{-it\lambda_k} \langle \lambda_k \rangle^\varepsilon \varphi_k(x), \quad
\langle D_t \rangle^\varepsilon e^{-itH} u_0 = \sum_k c_k \varphi_k(x) \langle \lambda_k \rangle^\varepsilon e^{-it\lambda_k}\end{equation*}
which imply \eqref{fractcal}.
\end{proof}

    \section{Energy estimates}   
  \label{sec:energy}
 In this section we establish a key energy estimate for solutions to
\[
\begin{cases}i\partial_t v- H v = \cos^{n-2}(2t) |e^{-it H}u_0+ v|^{2}( e^{-it H}u_0 +v),
\\v(0)=0
\end{cases}
\]
on the strip $(t,x)\in (-\frac \pi 4,\frac \pi 4)\times \R^n$, $n=2,3,4$
which is, as explained along the introduction, the main problem (at least for $n=2,3$ where the problem is energy subcritical) to be studied in order to get Theorem \ref{4d}. We denote
\begin{equation} \label{eq:xi} \xi(t,x)= (e^{-itH} u_0)(x). \end{equation}

If $n,s$ satisfy \eqref{eq:sn} then by Proposition \ref{sto} and \eqref{eq:khinchin} the following stochastic bounds holds for every $u_0 \in \Sigma_n^s$ with $K$ depending on $u_0$
(recall that $\Sigma_n^s$ is defined along Proposition \ref{sto}) : 
\begin{equation}\label{eq:stoch2} 
 K:=  \Vert \langle x \rangle   \xi \Vert_{L^\infty( \R\times \R^2)} +  \Vert \nabla \xi  \Vert_{L^\infty(\R\times \R^2)}+ \Vert \xi \Vert_{L^2([-\frac\pi4,\frac\pi4]\times \R^2)}  < \infty  \text{ if } n=2; \end{equation}
 and, 
 \begin{equation}\label{eq:stoch3} 
 \begin{split} 
K:= \, &  \Vert \langle x \rangle^{\frac65}   \xi \Vert_{L^\infty( \R\times \R^3)} + \Vert \langle x \rangle^{\frac15} \nabla \xi  \Vert_{L^\infty(\R\times \R^3)}  + \Vert \xi \Vert_{L^{\frac{12}{5}}([-\frac\pi4, \frac\pi4]\times \R^3)} 
\\ & +\Vert \langle x \rangle \xi \Vert_{L^{12}([-\frac{\pi}4,\frac\pi4]\times \R^3)}
+ \Vert \nabla  \xi \Vert_{L^{12}([-\frac{\pi}4,\frac\pi4]\times \R^3)} < \infty \qquad \text{ if } n=3; 
\end{split}
\end{equation}
and 
\begin{equation}\label{eq:stoch4} 
\begin{split} 
K:= \, & \Vert \langle x \rangle^{\frac75}   \xi \Vert_{L^\infty( \R\times \R^4)} + \Vert \langle x \rangle^{\frac25} \nabla \xi  \Vert_{L^\infty(\R\times \R^4)}
 +  \Vert \xi \Vert_{L^{\frac{8}{3}}([-\frac\pi4, \frac\pi4]\times \R^4)} 
 \\ & +  \Vert \langle x \rangle  \xi \Vert_{L^{8}([-\frac{\pi}4,\frac\pi4]\times \R^4)}   + \Vert \nabla \xi \Vert_{L^{8}([-\frac{\pi}4,\frac\pi4]\times \R^4)}   < \infty \qquad \text{ if } n=4. 
\end{split} 
\end{equation} 

Notice that the condition \eqref{eq:sn2} on $(n,s)$ in Proposition \ref{sto} is weaker than the condition \eqref{eq:sn} that we are assuming.   We write the equation as 
\begin{equation}\label{ener} 
\begin{cases}i\partial_t v- H v = \cos^{n-2}(2t) |\xi+ v|^{2}( \xi +v),
\\v(0)=0.
\end{cases}
\end{equation} 
We will provide energy estimates for weak solutions $v\in L^\infty([0,T], \mathcal{H}^1)$ to \eqref{ener}, with the obvious modification to $t \le 0$. By weak solution we mean a distributional solution. 
Observe that a straightforward calculation gives 
\[  \Vert Hw - |w+\xi|^2 (w+\xi) \Vert_{\mathcal{H}^{-1}} \le   c \Big( \Vert w \Vert_{\mathcal{H}^1}+
  \Vert w \Vert_{\mathcal{H}^1}^3+ K^3 \Big)\] 
and hence any weak solution defines a Lipschitz continuous map 
\[ (0,T) \to \mathcal{H}^{-1}(\R^n). \] 
In particular the limits of $v$ at the endpoints exist and by a slight abuse of notation we assume 
\[  v \in C_w ( [0,T] ; \mathcal{H}^1), \]
i.e. $v$ is weakly continuous on the closed interval. 
  
\begin{proposition}\label{energyestimate}
 Assume that $\xi$ satisfies
\eqref{eq:stoch2}, \eqref{eq:stoch3}, \eqref{eq:stoch4} respectively in dimension $n=2,3,4$ and 
let $v \in L^\infty([0, T); \mathcal{H}^1(\R^n))$ be weak solution to \eqref{ener} 
on the strip $[0, T)\times \R^n$, with $T\in [0, \frac \pi 4]$.
 Then 
\begin{equation}\label{enEST}\|v\|_{L^\infty([0,T); {\mathcal H}^1(\R^n))}\le K^4 e^{2\pi (K^2+1)}.\end{equation}
\end{proposition}

\begin{remark}
We will provide a formal proof, assuming as much regularity as we need, so that the idea becomes more clear. Working with weak solutions requires more care, but the arguments for this extension are standard.  
\end{remark} 

\begin{remark}\label{rem:energy}  The proof below applies to all dimensions with minor changes: If $ n \ge 3$, $ s \ge -\frac{n-2}4$, $ \gamma \in \mathcal{A}^s_n $ then for all $\varepsilon>0$ $\mu_\gamma$ a.s.
\[ \begin{split} \Vert \langle x \rangle^{\frac{n+2}4-\varepsilon}  \xi \Vert_{L^\infty(\R \times \R^n)} + \Vert \langle x \rangle^{\frac{n-2}4-\varepsilon}  \nabla \xi \Vert_{L^\infty(\R\times \R^n)} &
\\ & \hspace{-4cm}  + \Vert \xi \Vert_{L^{\frac{4n}{n+2}}( [-\frac\pi4,\frac\pi4] \times \R^n)} 
+ \Vert \nabla \xi \Vert_{L^{\frac{4n}{n-2}}([-\frac\pi4,\frac\pi4]\times \R^n)}  < \infty. \end{split} 
\] 
The same proof as for Proposition \ref{energyestimate}
also applies to (see Remark \ref{rem:sub}) 
\[ i \partial_t u - Hu = \cos^{\frac{(p-2)n}2-2} (2t) |u|^{p-2} u \] 
for $ 2+\frac4n \le  p \le 4$ for $ s \ge -\frac{n}2 +\frac{n+2}p$.  
For notational simplicity and clarity we restrict to $p=4$ and $2\le n \le 4$.
\end{remark}

\begin{proof}
We introduce the energy
\begin{equation*}
{\mathcal E}(v)=  \frac 12 \|v\|_{{\mathcal H}^1(\R^n)}^2 + \frac 14 \int_{\R^n} \cos(2t)^{n-2}|\xi+v|^{4} dx
\end{equation*}
and compute
\begin{multline*}\frac d{dt} {\mathcal E}(v(t,x))= \Re \int_{\R^n}  \partial_t v H\bar v dx +  \Re \int_{\R^n} \cos(2t)^{n-2}  \partial_t (\xi+v) (\overline{\xi+v}) |\xi+v|^{2}dx\\
- \frac{(n-2)}2\int_{\R^n} \sin (2t)\cos(2t)^{n-3 }|\xi+v|^{4} dx.\end{multline*}
where 
\[ \Re \int_{\R^n}  \partial_t v H\bar v dx =  -\Re \int_{\R^n} \partial_t v \big ( i \partial_t \bar v  + \cos^{n-2}(2t) |\xi +v |^2 (\overline{ \xi+v })\big )dx.
\]
The first term on the right hand side vanishes. 
Since $\sin (2t), \cos (2t)\geq 0$ for every $t\in [0,\frac \pi 4]$, by using the equation solved by $v$ and $\xi$ we get:
\[ 
\begin{split} 
\frac d{dt} {\mathcal E}(v(t,x))
\, & \leq \Re \int_{\R^n} \cos(2t)^{n-2} \big[- \partial_t v+ \partial_t ( v+\xi) \big] |\xi+v|^2 (\overline{\xi+v})  \, dx 
\\& =  \Re \int_{\R^n} \cos(2t)^{n-2} \partial_t \xi |\xi+v|^2 (\overline{\xi+v}) dx
\\& =\Im \int_{\R^n} \cos(2t)^{n-2} H\xi  |\xi+v|^2 (\overline{\xi+v}) dx
\end{split}
\] 
and by recalling that $H=-\Delta +|x|^2$ we can continue as follows
\[ 
\begin{split} 
\frac d{dt} {\mathcal E}(v(t))\, & \leq -\Im \int_{\R^n} \cos(2t)^{n-2} \Delta \xi  |\xi+v|^2 (\overline{\xi+v})dx
\\& \qquad +\Im \int_{\R^n} \cos(2t)^{n-2} |x|^2 \xi  |\xi+v|^2 (\overline{\xi+v})dx
\\ & =I+II.
\end{split}
\] 
Next we estimate $I$ by integration by parts and using the stochastic estimate \eqref{eq:stoch2}, \eqref{eq:stoch3}, \eqref{eq:stoch4}  as follows:
\[ 
|I| \leq 3\cos(2t)^{n-2}
\Big[ \int_{\R^n}  |\nabla \xi|^2  |\xi+v|^2 dx
+ \int_{\R^n}  |\nabla \xi| |\nabla v|  |\xi+v|^2 dx \Big] =: 3(I_1+ I_2),  
\]  
by Young's inequality, using  $\Vert \nabla \xi \Vert_{L^\infty} \le K $
\[ I_2 \le \frac{\cos^{n-2}(2t)}2 K\Big(  \Vert  \nabla v \Vert^2_{L^2}
+ \Vert v+\xi \Vert_{L^4}^4 \Big) \le 2 K \mathcal{E}(v(t)) \] 
and 
\[  I_1 \le   2 \int_{\R^n} |\nabla \xi|^2 |\xi|^2 dx   + 2\int_{\R^n} |\nabla \xi|^2 |v|^2 dx     \le  2\int_{\R^n} |\nabla \xi|^2 |\xi|^2 dx  + 2K^2  \Vert v \Vert^2_{L^2}.
\] 
Hence we get, using again the stochastic estimates (and that the $L^2$ part of the energy has not been used above) 
$$|I|\leq   (6K+ 6K^2)  {\mathcal E}(v(t)) +  6\int |\nabla \xi|^2 |\xi|^2 dx. $$
By a similar argument we get
$$|II|\leq 2K  {\mathcal E}(v(t)) +  \frac1K \int  |x |^2 |\xi|^4 dx .$$ 
We  apply  Gr\"onwall's Lemma and obtain for $0\le T \le \frac\pi4$ (using the $L^p$ bounds in \eqref{eq:stoch2}-\eqref{eq:stoch4}) 
\[ \sup_{0\le t \le T} \mathcal{E}(v(t)) \le e^{2\pi(K^2+1)} \int_0^T \int (|\nabla \xi|^2 +|x|^2 |\xi|^2)   |\xi|^2 dx\, dt \le K^4 e^{\pi(K^2+2)}.\] 
 \end{proof}

\begin{remark}\label{rem:sub}   Let $ 2< p < 4$. Then $u$ satisfies 
\[ i\partial_t u+ \Delta u = |u|^{p-2} u \] 
if and only if its lens transform $w$ satisfies 
\[ i \partial_\tau w - H w - \cos^{\frac{(p-2)n}2-2}  (2\tau) |w|^{p-2}  w = 0. \] 
 $p= 2+\frac4n$ is the mass critical exponent, for which  the exponent of $\cos(2\tau) $ vanishes. $p = 2+ \frac{4}{n-2}$ (for $n\ge 3$) is the energy critical case in which the eponent of $ \cos(2\tau) $ in the equation is $2$. Formally the exponent of $\cos$ becomes $-1$ if  $p= 2+ \frac2n$. 

The main difference in the case $ 2+\frac4n \le   p < 4$ is in the estimate 
\[ I:= \int_{\R^n} |\nabla \xi|\,  |\nabla v |\,  |v+\xi|^{p-2} dx \le \Vert \nabla \xi \Vert_{L^{\frac{2p}{4-p}}} \Vert  \nabla v \Vert_{L^2} 
\Vert v+\xi \Vert_{L^{p}}^{p-2}. 
\] 
Since 
\[ \Vert \nabla \xi \Vert_{L^{\frac{np}{n+2-p}}([-\frac\pi4, \frac{\pi}4]\times \R^n)}+ \Vert \nabla \xi \Vert_{L^\infty(\R\times \R^n)}  \le  K\] 
and $\frac{np}{n+2-p} \le \frac{2p}{4-p} $ for $ 2 <p \le 4$
we obtain by Young's inequality
\[ I \le K(t) \Vert \nabla v \Vert_{L^2} \Vert v+\xi \Vert_{L^p}^{p-2} \le K(t) \Big( \Vert \nabla v \Vert_{L^2}^2 + \Vert v+\xi \Vert_{L^p}^p \Big)^{\frac{3p-4}{2p}}\]
where 
\[ K(t):= \Vert \nabla \xi(t)\Vert_{L^{\frac{2p}{4-p}}(\R^n)}  \in L^{\frac{2p}{4-p}}([-\frac\pi4,\frac\pi4]).\] 
For $2+\frac2n<p<4$ we obtain a polynomial bound 
\[ \Vert v(t)  \Vert_{\mathcal{H}^1}\le c( 1+ K)^M \] 
for some positive $c$ and $M$.

\end{remark}

\section{Proof of Theorem \ref{4d} for \texorpdfstring{$n=2,3$}{n=2,3}}
 
 We are interested in solving the following Cauchy problems
 \begin{equation}\label{23}
 \begin{cases}
 i\partial_t v - H v =\cos_+^{n-2}(2t)|\xi+v|^2(\xi+v), 
 \\
 v(\tau)=\phi\in {\mathcal H}^1(\R^n)
 \end{cases}
 \end{equation}
 where $n=2,3$, $\tau\in \R $ and $(t, x)\in \R\times \R^n$ and $ \cos_+= |\cos|$.
  \begin{proposition}\label{nes}
Assume $\xi$ satisfies
\eqref{eq:stoch2}, \eqref{eq:stoch3} respectively for $n=2,3$. Let $K$ be the constant defined there.  
Then for every $R>0$ there exists $ \delta>0$ (a function of $R$ and $K$) such that  for every $\phi \in B_R({\mathcal{H}^{1}(\R^n))}$  and for every $\tau \in \R$ 
there exists a unique solution $v\in {\mathcal C}([\tau-\delta, \tau+\delta]; \mathcal{H}^1(\R^n))$ to \eqref{23}. 
\end{proposition} 

The 
globalization of the solution $v$ to \eqref{ener}  
on the full interval $[-\frac \pi 4, \frac \pi 4]$ readily comes by combining the local existence for \eqref{23} along with the uniform bound provided by Proposition \ref{energyestimate}, which of course can be extended to negative times as well.
As a consequence the Cauchy problem \eqref{ener} admits one unique global solution $v\in
{\mathcal C}([-\frac \pi 4, \frac \pi 4]; \mathcal{H}^1(\R^n))$ defined on the time interval  $[-\frac \pi 4, \frac \pi 4]$. 

We claim that  
the solution $v$ to \eqref{ener} has the following behaviour  
\begin{equation} \label{eq:contin}  \Vert v(t) -e^{-i(t\mp \frac{\pi}4)H} v(\pm \frac\pi4) \Vert_{\mathcal{H}^1(\R^n)} = \left\{ \begin{array}{cl}O(|\pm \frac\pi4-t|^{1-\delta}) 
\quad \forall \delta>0 & \text{ if } n=2 \\   O(|\pm \frac{\pi}4-t|^{\frac32}) \text{ if } n=3 
\end{array}
\right.
\end{equation} 
as $t\rightarrow \pm \frac \pi 4^\mp$.

Notice also that if $u_0$ belongs to the $\mu_\gamma$ full measure sets $\Sigma_n^s$ introduced along Proposition \ref{sto} 
for $n=2,3$ and $s,n$ as in \eqref{eq:sn}, then the corresponding free waves $\xi= e^{itH} u_0$ satisfy \eqref{eq:stoch2}, \eqref{eq:stoch3}.
Hence, based on the discussion done along the introduction, we get that Proposition \ref{nes} and estimate \eqref{eq:contin} imply Theorem \ref{4d}
in the cases $n=2,3$.
In particular recall that \eqref{eq:contin} implies 
\eqref{eq:convergence} for $n=2,3$ 
as discussed in the introduction. To shorten the notation we denote $I= [\tau-\delta, \tau+\delta]$. We have to prove Proposition \ref{nes} and inequality \eqref{eq:contin}. 

The numerology that we shall use will depend on the space dimension,
hence we shall split the proof in two cases, $n=2$ and $n=3$.
We rely on the standard contraction argument and Duhamel formulation, hence we shall not provide all the details. 
For simplicity we shall work for $t>0$, the argument for $t<0$ is the same.

\subsection{Proof of Prop. \ref{nes} for \texorpdfstring{$n=3$}{n=3}} 
\begin{proof}
We define the map $ \Xi$ by 
\begin{equation}\label{duham3} (\Xi (v))(t)= \int_\tau^{t} e^{-i(t-s)H} \cos(2s)|\xi(s)+ v(s)|^2(\xi(s)+ v(s)) ds. \end{equation}
By Strichartz estimates associated with the propagator $e^{itH}$ (see \eqref{strichD}) 
we get
 \[ \|\Xi v(t,x)  \|_{L^{\infty}(I; {\mathcal H}^1(\R^3))}
 \leq 
  C \|\cos (2s)|\xi+ v|^2(\xi+v) \|_{L^2(I; {\mathcal W}^{1, \frac65}(\R^3))}.
  \]
Next we estimate 
the nonlinear term for fixed time as follows:
\[  \| |\xi+v|^2(\xi+v)\|_{{\mathcal W}^{1, \frac 65}(\R^3))}
\\  \leq C  \| \nabla |\xi+v|^2(\xi+ v)\|_{L^\frac 65(\R^3)} + C  \| \langle x \rangle |\xi+v|^2 (\xi+ v)\|_{L^\frac 65(\R^3)}
\] 
hence  we get (since it does not matter of which term one takes the complex conjugate) 
\[ 
\begin{split}
\hspace{1cm} & \hspace{-1cm}    
\|\nabla |\xi+v|^2(\xi+v)\|_{L^\frac 65(\R^3)} \leq 3 \| (|v| +|\xi|)^2 (|\nabla v | + |\nabla \xi|) \|_{L^\frac 65(\R^3)}
\\ & \leq 6 \Big[ \|\nabla v\|_{L^2(\R^3)} (\|v\|_{L^6(\R^3)}^2 +  K^2) 
+ K\| v \|^2_{L^{\frac{12}{5}}} + CK^3 \Big] 
\\ &  \leq C(K^3+\|v\|_{\mathcal H^1(\R^3)}^3)
\end{split}
\] 
where we have used \eqref{eq:stoch3} and the Sobolev embedding ${\mathcal H}^1(\R^3)\subset L^2(\R^3)\cap L^6(\R^3)$.
Similarly one can prove
$$\||\langle x \rangle (v+\xi)| |v+\xi|^2\|_{L^\frac 65(\R^3)}
\leq C(K^3+\|v\|_{\mathcal H^1(\R^3)}^3)$$
and hence summarizing we get\
\[
\|\Xi v(t,x)\|_{L^{\infty}(I; {\mathcal H}^1(\R^3))}
\leq   C  \delta^{\frac12}(K^3+ \|v\|_{L^{\infty}(I; {\mathcal H}^1(\R^3))}^3). \] 
We search a fixed point of the map 
\[  v = e^{-i(t-\tau)H} \phi + \Xi(v). \]
Let $R = 2( \Vert \phi \Vert_{\mathcal{H}^1}+K) $ and 
\[ \delta = \frac{1}{ 9 C^2 R^4}. \] 
Then, if $ \Vert v \Vert_{L^\infty(I; \mathcal{H}^1)} < R$ 
\[  \Vert e^{-i(t-\tau)H} \phi + \Xi(v) \Vert_{L^\infty(I; \mathcal{H}^1)}
<  \frac{R}2 +   3C \delta^{\frac12}  R^3 \le   R.  \] 
Decreasing $\delta $ is necessary we obtain a strict contraction, and hence a unique solution. 
\end{proof}

Finally notice that by fixing the initial condition at $\tau=\frac \pi4$ as $\varphi=v(\frac \pi 4)$, 
by recalling that $v$ is a fixed point of the operator
\eqref{duham3} and by using the nonlinear estimate above we get with $I= (T,\frac{\pi}4)$
\[
\Big\|v(t,x)- e^{-i(t-\frac \pi 4)H} v(\frac \pi 4)  \Big\|_{L^{\infty}(I; {\mathcal H}^1(\R^3))}
\leq C(1+\|v\|_{L^{\infty}(I; {\mathcal H}^1(\R^3))}^3) \big (\int_{T}^{\frac \pi 4} \cos^2(2s) ds\big )^\frac 12
\]
and we conclude, since by Section \ref{sec:energy}  
$\max\limits_{t\in [0, \frac \pi 4]} \|v(t)\|_{{\mathcal H}^1}
<\infty$, the bound  
$$\Big\|v(t,x)- e^{-i(t-\frac \pi 4)H} v\big(\frac \pi 4\big)  \Big\|_{L^{\infty}((\frac \pi 4-T, \frac \pi 4); {\mathcal H}^1(\R^3))}\le C |T-\frac \pi 4|^\frac 32$$
and hence we obtain \eqref{eq:contin} for $n=3$.

\subsection{Proof of Prop. \ref{nes} for \texorpdfstring{$n=2$}{n=2}}
\begin{proof} 
The proof is similar to the case $n=3$ except that we have to adapt the Strichartz numerology.
We shall use the following
Strichartz estimate with any dual Strichartz pair such that $ \frac2p+\frac2q = 3 $, $ 1< p < 2$ in order to perform a fixed point argument: (recall the definition of $\Xi$ in \eqref{duham3}) 
\begin{equation}\label{duham2}  \| \Xi (v) \|_{L^{\infty}(I; {\mathcal H}^1(\R^2))}
\leq  
 C \||\xi+v|^2(\xi+v)\|_{L^p(I;{\mathcal W}^{1, q}(\R^2))}.
 \end{equation} 
 Then for every fixed time $t\in I$ we get:
 \[ 
 \begin{split}\hspace{1cm} & \hspace{-1cm} \|\nabla |\xi+v|^2(\xi+v)\|_{L^q(\R^2)}\leq  \| (|\xi|+ |v|)^2(|\nabla \xi| +| \nabla v|) \|_{L^q(\R^2)}
\\ &  \leq  \|\nabla v\|_{L^2} \|v\|_{L^\frac{4q}{2-q}(\R^2)}^2
+  K^3 + 3 K^2\|\nabla v\|_{L^2(\R^2)}
+ 3 K \|v\|_{L^{2q}(\R^2)}^2
\\ & \leq  C(K+\|v\|_{\mathcal H^1(\R^2)})^3
\end{split}
\] 
 where we have used the 
 Sobolev embedding $\mathcal H^1(\R^2)\subset L^r(\R^2)$ for every $r\in [2, \infty)$
 and the stochastic bound \eqref{eq:stoch2}.
Similarly we get
$$\|\langle x \rangle |\xi +v|^2(\xi+v)\|_{L^q(\R^2)}\leq C(K+\|v\|_{\mathcal H^1(\R^2)})^3$$
and hence by \eqref{duham2} 
 $$\|\Xi( v) \|_{L^{\infty}(I; {\mathcal H}^1(\R^2))}\leq C \delta^\frac 1p \| v \|^3_{L^{\infty}(I; {\mathcal H}^1(\R^2))} 
 $$
where $p\in (1, 2)$. We argue as in the case $n=3$ with a fixed point argument to obtain existence of a unique solutions by choosing $\delta$ small as a function of $K$ and $ \Vert \phi \Vert_{\mathcal{H}^1}$.
\end{proof}

Moreover by fixing, as in the case $n=3$ the initial condition at time $\tau=\frac \pi 4$ and $\varphi=v(\frac \pi 4)$
we get 
$$\| v(t,x)- e^{-i(t-\frac \pi 4) H} v(\frac \pi 4) \|_{L^{\infty}((T, \frac \pi 4); {\mathcal H}^1(\R^2))}=O(|T-\frac \pi 4|^\frac 1p)$$
and hence we conclude \eqref{eq:contin} for $n=2$ since $p$ can be choosen arbitrary close to $1$.

\section{Proof of Theorem \ref{4d} for \texorpdfstring{$n=4$}{n=4}}
We are interested in solving the following Cauchy problems 
 \begin{equation}\label{234tau}
 \begin{cases}
 i\partial_t v - H v =\cos^{2}(2t)|\xi+ v|^2(\xi+v), 
 \\
 v(\tau)=\phi\in {\mathcal H}^1(\R^4)
 \end{cases}
 \end{equation}
 where  $\tau\in \R_+$ and $(t, x)\in \R_+\times \R^4$. 
  \begin{proposition}\label{nes4}
Assume $\xi$ satisfies
\eqref{eq:stoch4} and let $K$ be the constant defined there. 
Then for every $R>0$ there exists $ \delta >0$ (a function of $R$ and $K$ ) such that  
for every $\phi \in B_R({\mathcal{H}^{1}(\R^4))}$,  and for every $\tau \in [0,\frac\pi4)$  there exists a unique solution $v\in {\mathcal C}([\tau, \min\{\frac\pi4,\tau+\delta\} ]; \mathcal{H}^1(\R^4))\cap L^4([\tau,\min\{\frac{\pi}4,\tau+\delta\}  ]; L^8(\R^4)) $ to \eqref{234tau}.
Moreover there exists a function $M:[0,\infty)\to [0,\infty)$ so that   
\[ \Vert v \Vert_{L^\infty((\tau,\min\{\frac\pi4,\tau+\delta\} ); \mathcal{H}^1)}+\Vert v \Vert_{L^4((\tau,\min\{\frac\pi4,\tau+\delta\} ); L^8)}
 \le M(\Vert \phi \Vert_{\mathcal{H}^1}). \] 
\end{proposition} 

As in the subcritical case 
we deduce that  the Cauchy problem \eqref{ener} admits one unique global solution $v\in
{\mathcal C}([-\frac \pi 4, \frac \pi 4]; \mathcal{H}^1(\R^4))$ defined on the strip $[-\frac \pi 4, \frac \pi 4]\times \R^4$. We claim that 
the solution $v$ to \eqref{ener} has the following behaviour  
\begin{equation} \label{eq:contin4}  \Vert v(t) -e^{-i(t\mp \frac{\pi}4)H} v(\pm \frac\pi4) \Vert_{\mathcal{H}^1(\R^4)} = O(|\pm \frac{\pi}4 -t|^2)  
\end{equation} 
as $t\rightarrow \pm \frac \pi 4$.  As in the subcritical case Theorem \ref{4d} follows from Proposition \ref{nes4} and   estimate \eqref{eq:contin4}. Our approach will be based on the  estimate for $  I=[\tau,\tau+\delta]$
  \begin{equation} \label{eq:strichartz}   \Vert u \Vert_{L^\infty(I; \mathcal{H}^1(\R^4))} + \Vert u \Vert_{L^4(I; L^8(\R^4))} \lesssim   \Vert u(\tau ) \Vert_{\mathcal{H}^1(\R^4)} +  \Vert i\partial_t u - Hu \Vert_{L^2(I; \mathcal{W}^{1,\frac43}(\R^4)) },  
 \end{equation} 
 which is a consequence of the Strichartz estimates combined with the Sobolev inequality $ \Vert u(t) \Vert_{L^8(\R^4)} \leq C \Vert \nabla u(t) \Vert_{L^{\frac83}(\R^4)} $.
We  assume that $\xi= e^{itH} u_0$ satisfies \eqref{eq:stoch4} and $K$ will be the  quantity defined therein.
Due to Proposition \ref{sto}
this property is true almost surely w.r.t. $\mu_\gamma$ by assuming the condition \eqref{eq:sn} for $n=4$. We shall split the proof of Proposition \ref{nes4} into several subsections.

\subsection{The global existence and bounds by Ryckman and Visan}

The crucial ingredient is the result of Ryckman and Visan \cite{RV} about global existence and bounds for the equation without noise.  The following lemma is a consequence. 
 \begin{lemme}\label{approxbu}
For every  $\tau\in (0, \frac{\pi}4)$ and $ \phi \in \mathcal{H}^1(\R^4)$ there exists a solution $$w \in {\mathcal C}([\tau, \frac \pi 4); {\mathcal H}^1(\R^4))\cap L^4((\tau, \frac \pi 4);L^8(\R^4)),$$  to 
the Cauchy problem
\begin{equation}\label{apporxim}
\begin{cases}
i\partial_t w - H w= \cos^2(2t) w|w|^2, \quad (t, x)\in [\tau, \frac \pi4)\times \R^4\\
w(\tau)= \phi.
\end{cases}
\end{equation}
Moreover  there exists a monotone function $M: [0,\infty) \to [0,\infty)$ so that 
\begin{equation}\label{unifL6bound}
\|w\|_{L^4((\tau, \frac \pi 4);L^8(\R^4))}\le M( \Vert \phi\Vert_{\mathcal{H}^1})  \end{equation}
and 
\begin{equation}\label{uniformH1bound}
\|w\|^2_{L^\infty((\tau, \frac \pi 4); {\mathcal H}^1(\R^4))} \le   \Vert \phi \Vert^2_{\mathcal{H}^1} +  \frac12 \Vert \phi \Vert^4_{L^4}  .\end{equation}
\end{lemme} 
\begin{proof}{\bf Step 1. Wellposedness on an interval whose length depends on $ \phi$, not on its norm.}
We claim that there exists a unique weak solution on a small time interval. 
Let $ u = e^{-i(t-\tau)H} \phi \in L^\infty(\R ; \mathcal{H}^1) \cap L^4( I; L^8( \R^4)) $. We make the Ansatz $w = v+ u$ where $v(\tau)=0$,  
\[ i\partial_t  v - H v = \cos^2(2t)  |v+u|^2 (v+u). \] 
and 
\[ \Xi(v)= \int^t_\tau  e^{-i (t-s) H}\cos^2(2s) |v+u|^2 (v-u) ds. \] 
We define 
\[ \vvvert f \vvvert =  \Vert f \Vert_{L^\infty(I;\mathcal{H}^1)} + \Vert f \Vert_{L^4(I; L^8(\R^4))}. \] 
For $\delta >0$ and $I= [\tau,\tau+\delta]$  we obtain with absolut implicit constants
\[ 
\begin{split} 
\vvvert \Xi(v) \vvvert \, &   \lesssim \Vert |v+u|^2 (v+u) \Vert_{L^2(I; \mathcal{W}^{1,\frac43}) } 
\\ &  \lesssim   (\Vert v \Vert^2_{L^4(I; L^8(\R^4))} +\Vert u \Vert^2_{L^4(I; L^8(\R^4))}) ( \Vert v \Vert_{L^\infty(I; \mathcal{H}^1)} +\Vert u \Vert_{L^\infty(I; \mathcal{H}^1) }).
\\ & \lesssim (\vvvert v \vvvert + \Vert u \Vert_{L^4(I; L^8(\R^4))})^2 (\vvvert v \vvvert + \Vert \phi \Vert_{\mathcal{H}^1} ). 
\end{split} 
\] 
We may choose $\delta $ small so that $\Vert u \Vert_{L^4([\tau, \tau+\delta]; L^8(\R^4))}$ becomes small. It is now easy to construct a solution by a fixed point argument. This time $\delta$ depends on $\phi$, and not only its norm. The same argument can be applied to the left of $ \tau$ and we obtain a solution on an interval $(\tau-\delta,\tau+\delta)$.
More precisely, if we take the factor $\cos^2(2\tau) $
 into account  we obtain an existence intervall $[\tau ,\frac{\pi} 4)$
provided 
\[ \Vert \cos(2.) u \Vert_{L^4([\tau,\pi/4); L^8} \le c \cos(2\tau) \Vert \phi \Vert_{\mathcal{H}^1} \]
is sufficiently small.

 \bigskip

\noindent{\bf Step 2: The energy estimate.}
We next  establish \eqref{uniformH1bound} on the interval $[\tau, T)$, assuming 
that 
\[ w \in L^\infty([\tau,T]; \mathcal{H}^1(\R^4)) \] 
is a weak solution. This is a simpler version of the energy estimate in Section \ref{sec:energy}. Again we argue formally and calculate
\[\frac{d}{dt} \big(\frac12\|w(t)\|_{{\mathcal H}^1(\R^4)}^2 + \frac{\cos^2(2t)}4 \|w(t)\|_{L^4(\R^4)}^4\big)  =-\sin (2t) \cos(2t)\|w(t)\|_{L^4(\R^4)}^4, \]
hence the energy is monotonically decreasing and thus \eqref{uniformH1bound} holds. 

\bigskip 

\noindent{\bf Step 3. Consequences of the global estimate by Ryckman and Visan.}

Next we consider the solution to the following classical NLS on $\R^4$ (which corresponds to \eqref{apporxim}
up to inverse lens transform):
\begin{equation}\label{pseudoconflens}\begin{cases}
i\partial_t u + \Delta u= u|u|^2, \quad (t,y)\in \R\times \R^4\\
u(\frac{\tan (2\tau)}2,y)= e^{i\frac{\sin(2\tau)\cos(2\tau)}2|y|^2}\cos^2(2\tau) \phi( \cos(2\tau) y)\end{cases}
\end{equation}
and observe 
\[
\Big\Vert \nabla u (\frac{\tan (2\tau)}2)  \Big\Vert_{L^2(\R^4)}^2 \leq 2 \cos^2(2\tau)  \Vert \nabla \phi\Vert_{L^2}^2  + 2  \sin^2(2\tau)  
\Vert y \phi  \Vert_{L^2(\R^4)}^2 
\le  2\Vert \phi  \Vert^2_{\mathcal{H}^1}. 
\]
Then by the analysis in \cite{RV} (more precisely see Lemma 3.6 in \cite{RV}) of the cubic NLS on $\R^4$ we get that
\eqref{pseudoconflens} admits one unique global solution with the following uniform bound for some monotone function $\tilde M$
\begin{equation}\label{L6utau}      \|u \|_{L^4 (\R; L^8(\R^4))}\le \tilde M( \Vert \nabla \phi \Vert_{L^2}).\end{equation}
Moreover we have by direct computation
\begin{equation}\label{estimL4L8}\|u\|_{L^4(
\frac{\tan(2\tau)}2, \frac{\tan(2T)}2);L^8(\R^4))}=\|\cos(2t) w \|_{L^4((\tau,T);L^8(\R^4))}.
\end{equation}

\bigskip 

\noindent{\bf Step 4: Conclusion.}
Let $\tau < T \le \frac{\pi}4$ and assume that 
 \begin{equation} \label{eq:conclusion}  w \in C([\tau, T); \mathcal{H}^1) \cap \bigcap_{S< T} L^4( [\tau, S); L^8(\R^4))\end{equation}
 is a weak solution and that $T$ is the maximal time for which a weak solution exists. Given any $ t \in [\tau,T]$ there is a unique solution with initial datum $\phi= w(t)$, hence this weak solution is unique. By Step 2 $\Vert w(t) \Vert_{\mathcal{H}^1}$ is uniformly bounded. Inverting the lens transform we obtain a weak solution $u$ to 
 \[ i\partial_t u + \Delta u = |u|^2 u \]
which is (by the $L^4 L^8$ bound)  the solution of Ryckman and Visan 
and hence 
\[ \Vert \cos(t) w \Vert_{L^4([\tau, T]; L^8)} \le M(\Vert \phi \Vert_{\mathcal{H}^1}). \]
As discussed in Section \ref{sec:energy}  the weak limit 
\[   \lim_{ t\to T} w(t) =: \psi \in \mathcal{H}^1 \] 
exists.
If $ T < \frac{\pi}4 $ we apply Step 1 and obtain a weak solution on an interval $[T,T+\delta)$ which we can combine to a weak solution on 
$[\tau, T+\delta)$, contradicting the choice of $T$.
We have thus obtained a unique weak solution $w$ with $T=\frac{\pi}4$ in \eqref{eq:conclusion}. By the remark at the end of step 1 we obtain for some monotone function $\tilde M$
\[ \Vert w \Vert_{L^4([\tau, \pi/4); L^8(\R^4))} \le \tilde M(\Vert \phi\Vert_{\mathcal{H}^1}).\]
\end{proof}

\subsection{A linearized problem} 

We shall need the following lemma as well.
\begin{lemme}\label{stet} Let $0 \le \tau < \tau+\delta \le \frac{\pi}4$, let
$I= (\tau,\tau+\delta)$, $ \phi \in \mathcal{H}^1$, $F \in L^2(I; \mathcal{W}^{1,\frac43})$,  
\[ \Vert w \Vert_{L^\infty([\tau, \tau+\delta];\mathcal{H}^1) } \le  L \] 
and 
\[ \Vert w \Vert_{L^4([\tau,\tau+\delta];L^8(\R^4))} \le M.  \] 
 Then  
the following problem
$$\begin{cases}
i\partial_t h- Hh -(\cos(2t))^2 (2h|w+\xi|^2+ \bar h (w+\xi)^2)= F\\
h(\tau)=\phi \in {\mathcal H}^1(\R^4)
\end{cases}
$$
has a unique solution $h \in C(I; \mathcal{H}^1) \cap L^4(I; L^8(\R^4))$. In addition 
\begin{equation}\label{StRi}\|h\|_{L^\infty( I;\mathcal H^1)}+\|h\|_{L^4(I;L^8(\R^4))} \leq C(K,L,M )  \big( \|\phi\|_{{\mathcal H}^1(\R^4)} + \|F\|_{{L^{2}(I;{\mathcal W}^{1,\frac 43}(\R^4)})}\big). 
\end{equation}
\end{lemme}

\begin{proof}
It suffices to prove this lemma  for some small $\tau$ depending on $K$, $L$ and $M$, since we then simply iterate the estimate.  We will choose $\tau$ below. 

By the Strichartz estimate, with an absolute implicit constant, and using that complex conjugates do not affect the estimate, using again $\vvvert . \vvvert$, 
\begin{equation} \label{strichhomgmn}
\vvvert h\vvvert 
\lesssim  \|\phi\|_{{\mathcal H}^1(\R^4)}+  \| |w +\xi|^2 h \|_{L^2(I;{\mathcal W}^{1,\frac 43}(\R^4))}
+  \|F\|_{L^2(I;{\mathcal W}^{1,\frac 43}(\R^4))}.
\end{equation} 
We estimate the terms on the right hand side one by one
\begin{equation} 
\begin{split} \hspace{2cm} & \hspace{-2cm} 
\| \langle x \rangle ( |w+\xi|^2 h ) \|_{L^2(I; L^{\frac43}(\R^4))}  +
\|\nabla  ( |w+\xi|^2 h ) \|_{L^2(I;L^{\frac 43}(\R^4))}
\\ \le  \, & 
C \Big( \delta^{\frac12} K^2+  (\delta^{\frac14} K+ L+M)  \Vert w \Vert_{L^4(I;L^8(\R^4))} \Big) 
\vvvert  h  \vvvert. 
\end{split} 
\end{equation}

Hence going back to \eqref{strichhomgmn} we get
\begin{equation} \label{strichhomgr}
\begin{split} \hspace{2cm} & \hspace{-2cm} 
\vvvert h\vvvert
\lesssim  \Vert \phi \Vert_{\mathcal{H}^1}   +  \|F\|_{L^{2}(I;{\mathcal W}^{1,\frac 43}(\R^4))}
\\  & +
\Big( \delta^{\frac12}  K^2  
+ \Big(\delta^{\frac14} K   
 +   L +M \Big) \Vert w \Vert_{L^4(I;L^8(\R^4))} \Big) 
\vvvert  h  \vvvert.
\end{split}  
\end{equation} 
We fix $\delta>0$ so small that we can absorb the  term with $ \delta^{\frac12}$ (as a function of $K$) on the left hand side. Then with $C$ depending on $K$, $L$ and $M$,
\begin{equation} \label{strichhomgrr}\vvvert h \vvvert \le C\Big(  \|\phi\|_{{\mathcal H}^1(\R^4)}+
\|w\|_{L^{4}(I;L^8(\R^4))} \vvvert h \vvvert 
+   \|F\|_{L^{2}(I;{\mathcal W}^{1,\frac 43}(\R^4))} \Big).
\end{equation}  
As discussed above it suffices to prove the statement for small $\tau$ depending on $K$, $L$ and  $M$. 
If $  \Vert w \Vert_{L^4(I; L^8)} \le \frac1{2C}$ 
we can absorb the middle term on the left hand side and arrive at 
\[ \vvvert h \vvvert \le C (\Vert \phi \Vert_{\mathcal{H}^1} + \Vert F \Vert_{L^2(I; \mathcal{W}^{1,\frac43})}). \]
In this case it is easy to turn this estimate into an iteration procedure to construct a unique solution.

We know from Lemma \ref{approxbu}  that 
$ \Vert w \Vert_{L^4([\tau, \frac{\pi}4); L^8(\R^4))} \le \tilde M(\phi)$,
hence we can split $I$ into at most $ N=        (2C \tilde M(\Vert \phi \Vert_{\mathcal{H}^1}))^{\frac14} $
intervals $I_j= [\tau_j,\tau_{j+1}]$,
\[ \tau = \tau_0 < \tau_1 < \dots <\tau_N = \tau+\delta \] 
so that $ \Vert w \Vert_{L^4(I_j)} \le \frac1{2C}$.  
We repeat the previous estimate $N$ times on these small intervals 
\[ \Vert h \Vert_{L^\infty(I_j;\mathcal{H}^1)} + \Vert h \Vert_{L^4(I_j; L^8(\R^4))}
\le C \big( \Vert h(\tau_{j}) \Vert_{\mathcal{H}^1} 
+ \Vert F \Vert_{L^2(I_j; \mathcal{W}^{1,\frac43})}\big)\] 
and obtain 
\[ \Vert h \Vert_{L^\infty( I; \mathcal{H}^1)}+ \Vert h \Vert_{L^4( I; L^8)}
\le C^N \big( \Vert \phi \Vert_{\mathcal{H}^1} + N^\frac12 \Vert F \Vert_{L^2(I; \mathcal{W}^{1,\frac43})}\big). \]  
\end{proof} 

\subsection{Proof of Proposition \ref{nes4}}

We consider with $ \xi $ satisfying \eqref{eq:stoch4}
the equation 
 \begin{equation}\label{234}
 \begin{cases}
 i\partial_t v - H v =\cos^{2}(2t)|\xi+v|^2(\xi+v), 
 \\
 v(\tau)=\phi\in {\mathcal H}^1(\R^4)
 \end{cases}
 \end{equation}
Proposition \ref{nes4} states that there exists $ \delta >0$ depending only on $K$ and 
$\Vert \phi \Vert_{\mathcal{H}^1} $ so that there is a unique weak solution 
\[   v \in C([\tau, \tau+\delta];\mathcal{H}^1) \cap L^4([\tau, \tau+\delta]; L^8(\R^4)), \]
which is the statement we prove in this subsection. 

 \begin{proof}
Let $w\in C([ \tau, \frac{\pi}4 ]; \mathcal{H}^1(\R^4)) \cap L^4([\tau, \frac{\pi}4];L^8(\R^4))$ be the solution considered in Lemma \ref{approxbu}. It satisfies a bound
\[  \Vert w \Vert_{L^\infty([\tau, \frac{\pi}4];\mathcal{H}^1)} + 
\Vert w \Vert_{L^4([\tau, \frac{\pi}4]; L^8(\R^4))} \le \tilde M( \Vert \phi \Vert_{\mathcal{H}^1}). \]

We look for the equation solved by $v-w=R$, 
 $$
 \begin{cases}
 i\partial_t R -HR=\cos^2(2t)\big[(R+w+\xi)|R+w+\xi|^2-w|w|^2\big]\\
 R(\tau)=0.
 \end{cases}$$
We expand the nonlinearity and we get
\begin{equation}\label{eqR} 
i\partial_t R -HR-\cos^2(2t)\big( \bar R (w+\xi)^2+2R |w+\xi|^2\big)= F 
\end{equation} 
with 
\begin{equation} 
 F(R) =  \cos^2(2t)\Big[
R^2 \bar R+ R^2 (\overline{w+\xi})+2|R|^2 (w+\xi)+
\big(|w+\xi|^2(w+\xi)-w|w|^2\big)\Big].
 \end{equation} 
The important feature are 
\begin{enumerate} 
\item The left hand side is a linear equation for $R$ considered in the previous subsection. 
\item The term $|w|^2w$ cancels on the right hand side. 
\end{enumerate}
We will construct the solution by a standard fix point argument 
\[ R = \Xi(R) \] 
where $ r = \Xi(R) $ is the unique solution to the linear equation  
 \begin{equation}\label{eqtR}
i\partial_t r -H  r-\cos^2(2t)\big( \bar r (w+\xi)^2+2r |w+\xi|^2\big) = F(R) 
\end{equation} 
with initial data $ r(0)=0$.

We aim to apply Lemma \ref{stet} on the time interval $[\tau, \tau+\delta]$. Notice that the forcing term in \eqref{eqR} involves terms quadratic and cubic in $R$
as well as a remainder independent of $R$. 
We first  compute the terms quadratic in $R$ by neglecting as usual the bounded factor $\cos^2(2t)$, in particular we shall estimate $\|R^2 \bar w\|_{L^2([\tau, \tau+\delta]; \mathcal W^{1,\frac43}(\R^4))}$ and
$\|R^2 \overline{\xi}\|_{L^2
([\tau, \tau+\delta];\mathcal W^{1,\frac43}(\R^4))}$. To shorten the notation we denote $I = [\tau,\tau+\delta]$. In the same fashion as above  
\begin{equation}   \||R|^2 w\|_{L^2(I;{\mathcal W}^{1,\frac 43}(\R^4))} +
\|R^2 \bar w \|_{L^2(I;\mathcal{W}^{1,\frac 43}(\R^4))}
  \leq C\vvvert  w\vvvert \vvvert R \vvvert^2. 
\end{equation}  
Next notice that
\begin{equation} \|R^2 \overline{\xi} \|_{L^2(I;\mathcal{W}^{1,\frac43})}
+  \||R|^2 \xi \|_{L^2(I;\mathcal{W}^{1,\frac43})}
\leq C \delta^{\frac14} K  \vvvert R \vvvert^2.\end{equation}  
where we used \eqref{eq:stoch4}. 
Next we estimate the term cubic in $R$:
\begin{equation}\label{aslcubic}
\|\bar R R^2\|_{L^2(I; {\mathcal W}^{1,\frac 43}(\R^4))}\\
\leq C \|R\|_{L^\infty(I;{\mathcal H}^{1}(\R^4))}\|R\|_{L^4(I;L^8(\R^4))}^2
\le C \vvvert R \vvvert^3.
\end{equation}
Finally we compute the last term that appears in the forcing term of \eqref{eqR}. It is crucial that there is no cubic term  of $w$. We obtain (using that the position of complex conjugates does not matter) 
\begin{equation} \begin{split} \label{Forcing} 
\||w+\xi|^2(w+\xi)-w|w|^2\|_{L^2(I;{\mathcal W}^{1,\frac 43})} & \\ & \hspace{-4cm}  \leq    \|\xi|^2\xi \|_{L^2(I;{\mathcal W}^{1,\frac 43})}
 + 3\| |\xi|^2 w\|_{L^2(I;{\mathcal W}^{1,\frac 43})}
+ 3\| |w|^2 \xi\|_{L^2(I;{\mathcal W}^{1,\frac 43})}
\end{split} 
\end{equation} 
and, estimating the terms on the right hand side one by one 
\begin{equation}   \| |\xi|^2\xi\|_{L^2(I;{\mathcal W}^{1,\frac 43})} \le   C\delta^{1/2} K^3, \end{equation} 
\begin{equation} 
\| |\xi|^2 w\|_{L^2(I;{\mathcal W}^{1,\frac 43})}
\le c K^2 \delta^{\frac12} \Vert w \Vert_{L^\infty(I; \mathcal{H}^1)} 
\end{equation} 
\begin{equation} 
\| |w|^2 \xi\|_{L^2(I;{\mathcal W}^{1,\frac 43})}
\le c \delta^{\frac14} K \Vert w \Vert_{L^4(I; L^8)} \Vert w \Vert_{L^\infty(I; \mathcal{H}^1)}.  
\end{equation} 

By combining \eqref{eq:stoch4} with the bounds \eqref{unifL6bound} and \eqref{uniformH1bound}
we get:
\begin{equation}\label{Ftau}
\||w+\xi|^2(w+\xi)-w|w|^2\|_{L^2(I;{\mathcal W}^{1,\frac 43}(\R^4))}
\leq C(K, \Vert \phi \Vert_{\mathcal{H}^1} ) \delta^\frac 14. 
\end{equation}
Gathering together the estimates proved above, 
we apply   Lemma \ref{stet} and  we get (with implicit  constants depending only on $K$ and $ \Vert \phi \Vert_{\mathcal{H}^1}$) 
\begin{equation}
\vvvert \Xi(R) \vvvert   \lesssim  \vvvert R \vvvert^2   +\vvvert R \vvvert^3 
+ \delta^{\frac14}. 
\end{equation} 
and, by polarization 
\[ 
\vvvert \Xi(R_2)- \Xi(R_1) \ \vvvert   \lesssim  \Big(\vvvert R_2 \vvvert + \vvvert R_1 \vvvert + \vvvert R_2 \vvvert^2+ \vvvert R_1 \vvvert^2 \Big)  \vvvert R_2-R_1 \vvvert. 
\] 
If $\delta$ is sufficiently small (depending only on $K$ and $\Vert \phi \Vert_{\mathcal{H}^1}$) standard fixpoint arguments yield a unique solution.  
\end{proof} 

\subsection{Inequality \texorpdfstring{\eqref{eq:contin4}}{}} 

We complete the proof by verifying \eqref{eq:contin4} and recall that we have constructed a solution $v$  to 
\[ i \partial_t v - Hv =\cos^2(2t)  |\xi+v|^2(\xi+v) \] 
where $\xi$ satisfies 
\[ i \partial_t \xi-H\xi = 0 \] 
and the bounds \eqref{eq:stoch4}. Moreover, by Proposition \ref{nes4}     
\[  \Vert v \Vert_{C([0,\frac{\pi}4];\mathcal{H}^1)}  +
 \Vert v \Vert_{L^4((0,\frac{\pi}4); L^8(\R^4))} \le C(K).   \]  
Thus
\[   v(t)  - e^{-i(t-\frac{\pi}4)H} v(\pi/4) =  -\int_t^{\frac{\pi}4} 
e^{-i(t-s)H} \cos^2(2s)  |\xi(s)+v(s)|^2(\xi(s)+v(s) ) ds\]
By the Strichartz estimate
\[ 
\begin{split} \hspace{1cm} & \hspace{-1cm} 
\Big\Vert  \int_t^{\frac{\pi}4} 
e^{-i(t-s)H} \cos^2(2s)  |\xi(s)+v(s)|^2(\xi(s)+v(s) ) ds \Big\Vert_{L^2([\frac{\pi}4-\delta,\frac{\pi}4]; \mathcal{W}^{1,\frac43} )} 
\\ & \le C \cos^2\big(\frac{\pi}2-2\delta\big) 
\Big\Vert  |\xi+v|^2(\xi+v) \Big\Vert_{L^2( [\frac{\pi}4-\delta,\frac{\pi}4]; \mathcal{W}^{\frac43}) }
\\ & \le 4 C \delta^2 ( K^3 + \vvvert v \vvvert^3).  
\end{split} 
\]

\section{Higher dimensions} 
\subsection{The energy critical case for \texorpdfstring{$n\ge 5$}{n>=5}}  
There are only minor changes in the energy critical case $n\ge 5$ and $p = \frac{2n}{n-2}$.
\begin{enumerate} 
\item  The energy bound has been discussed in Remark \ref{rem:energy}.
\item The Strichartz estimate is 
\[\begin{split} \hspace{2cm} & \hspace{-2cm}  \Vert v \Vert_{L^\infty ([-\frac\pi4,\frac\pi4]; \mathcal{H}^1( \R^n))} 
+ \Vert  v \Vert_{L^{2}([-\frac\pi4,\frac\pi4];\mathcal{W}^{1,\frac{2n}{n-2}}(\R^n))}\\ &  \le c \Big[ \Vert v(0) \Vert_{\mathcal{H}^1} + \Vert i\partial_t v - Hv \Vert_{L^2([-\frac\pi4,\frac\pi4]; \mathcal{W}^{1,\frac{2n}{n+2}}(\R^n))}\Big]. 
\end{split} 
\] 
and we use the Sobolev estimate
\[ \Vert v(t)  \Vert_{L^{\frac{2n}{n-2}}(\R^n)} \le C \Vert v(t) \Vert_{\mathcal{H}^1(\R^n)}. \] 
\item For the deterministic equation wellposedness and uniform bounds for 
\[ \Vert w \Vert_{L^{2 \frac{n+2}{n-2}}(\R\times \R^n)} \]  
 have been proven in \cite{Vi} (and hence for all Strichartz norms) 
  replacing the $L^4 L^8$ bound. 
\item We now complete the proof in the same fashion as above.
\end{enumerate} 

\subsection{The mass supercritical and energy subcritical case \texorpdfstring{$2+\frac4n \le  p < 2+ \frac{4}{n-2}$ and $n\geq 5$}{2+2/n < p <= 2+ 4/(n-2)}.} 

This follows from the energy estimate and easier local existence results. 

\bigskip 

\noindent {\bf Acknowledgements} H.K. was partially supported by
the Deutsche Forschungsgemeinschaft (DFG, German Research Foundation) through the
Hausdorff Center for Mathematics under Germany’s Excellence Strategy - EXC-2047/1 -
390685813 and through CRC 1060 - project number 211504053.
N.T. was partially supported by the ANR project Smooth ANR-22-CE40-0017.
N.V. is supported by the PRIN project 2020XB3EFL by MIUR and PRA-2022-11 by the University of Pisa. We thank Shao Liu for the feedback and carefully reading the paper.  

\bibliographystyle{plain}

\bibliography{references.bib}



\end{document}